\documentclass{article}
\usepackage{amsfonts}
\usepackage{amssymb}
\usepackage{amsmath}
\usepackage{amsthm}

\newtheorem{theorem}{Theorem}
\newtheorem{proposition}{Proposition}

\newtheorem{lemma}{Lemma}
\newtheorem{remark}{Remark}
\newtheorem{corollary}{Corollary}

\renewcommand{\Re}{\mathop{\mathrm{Re}}}

\newcommand{\erfc}{\mathop{\mathrm{erfc}}\nolimits}

\newcommand{\var}{\mathop{\mathrm{var}}}

\newcommand{\Lip}{\mathop{\mathrm{Lip}}}

\newcommand{\SO}{\mathop{\mathrm{SO}}}
\newcommand{\OO}{\mathop{\mathrm{O{}}}}

\newcommand{\dist}{\mathop{\mathrm{dist}}}

\newcommand{\CV}{\mathop{\mathrm{CV}}}

\newcommand{\one}{{\mathord{\mathbf1}}}

\newcommand{\cD}{\mathord{\mathcal{D}}}

\newcommand{\cE}{{\mathord{\mathcal{E}}}}

\newcommand{\cH}{\mathord{\mathcal{H}}}

\newcommand{\cL}{\mathord{\mathcal{L}}}

\newcommand{\cS}{{\mathord{\mathcal{S}}}}
\newcommand{\cU}{{\mathord{\mathcal{U}}}}

\newcommand{\sfE}{{\mathord{\mathsf{E}}}}

\newcommand{\sfM}{{\mathord{\mathsf{M}}}}

\newcommand{\sfBC}{\mathord{\mathsf{BC}}}

\newcommand{\bfE}{{\mathord{\mathbf{E}}}}

\newcommand{\bfu}{{\mathord{\mathbf{u}}}}

\newcommand{\bfS}{{\mathord{\mathbf{S}}}}
\newcommand{\bft}{{\mathord{\mathbf{t}}}}

\newcommand{\frh}{\mathord{\mathfrak{h}}}

\newcommand{\frg}{\mathord{\mathfrak{g}}}

\newcommand{\frl}{\mathord{\mathfrak{l}}}

\newcommand{\bbC}{{\mathord{\mathbb{C}}}}

\newcommand{\bbN}{{\mathord{\mathbb{N}}}}
\newcommand{\bbP}{{\mathord{\mathbb{P}}}}
\newcommand{\bbR}{{\mathord{\mathbb{R}}}}

\newcommand{\bbT}{{\mathord{\mathbb{T}}}}
\newcommand{\bbZ}{{\mathord{\mathbb{Z}}}}

\newcommand{\al}{{\mathord{\alpha}}}

\newcommand{\ga}{{\mathord{\gamma}}}
\newcommand{\Ga}{{\mathord{\Gamma}}}
\newcommand{\vf}{{\mathord{\varphi}}}

\newcommand{\si}{{\mathord{\sigma}}}
\newcommand{\ka}{{\mathord{\kappa}}}
\newcommand{\vk}{{\mathord{\varkappa}}}

\newcommand{\la}{{\mathord{\lambda}}}

\newcommand{\De}{{\mathord{\Delta}}}

\newcommand{\ep}{{\mathord{\varepsilon}}}

\newcommand{\Tr}{\mathop{\mathrm{Tr}}\nolimits}

\newcommand{\scal}[2]{\left<#1,#2\right>}

\let\td=\tilde

\newcounter{ccc}
\date{}
\title{\bf Metric properties in the mean of polynomials on
compact isotropy irreducible homogeneous spaces}
\author{V.M. Gichev\thanks{Part of the work was done during my stay
in the Institut Mittag-Leffler (Djursholm, Sweden), 2011 fall. I
thank the Institut for support and hospitality.}}
\begin{document}
\maketitle
\begin{abstract}
Let $M=G/H$ be a compact connected isotropy irreducible Riemannian
homogeneous manifold, where $G$ is a compact Lie group (may be,
disconnected) acting on $M$ by isometries. This class includes all
compact irreducible Riemannian symmetric spaces and, for example,
the tori $\bbR^n/\bbZ^n$ with the natural action on itself
extended by the finite group generated by all permutations of the
coordinates and inversions in circle factors. We say that $u$ is a
polynomial on $M$ if it belongs to some $G$-invariant finite
dimensional subspace $\cE$ of $L^2(M)$. We compute or estimate
from above the averages over the unit sphere $\cS$ in $\cE$ for
some metric quantities such as Hausdorff measures of level set and
norms in $L^p(M)$, $1\leq p\leq\infty$, where $M$ is equipped with
the invariant probability measure. For example, the averages over
$\cS$ of $\|u\|_{L^p(M)}$, $p\geq2$, are less than
$\sqrt{\frac{p+1}{e}}$ independently of $M$ and $\cE$.
\end{abstract}


\section{Introduction}

Let $M$ be a compact connected Riemannian manifold, $G$ be a
compact Lie group acting on $M$ transitively by isometries, and
$H$ be the stable subgroup of a base point $o\in M$. We assume
that $M$ is isotropy irreducible, i.e., that the group $H$ has no
proper invariant subspaces in $T_oM$ {  (the action of $H$ is
induced by its adjoint representation in the Lie algebra $\frg$ of
$G$)}. This class of homogeneous spaces is rather wide
--- it includes all irreducible Riemannian symmetric spaces, in
particular, real spheres, Grassman manifolds, and simple Lie
groups. The mentioned spaces are strongly isotropy irreducible,
i.e., the connected component of $H$ is irreducible in $T_oM$. A
torus $T$ considered as a homogeneous space of the semidirect
product of $T$ and a finite group $F$ of its isometrical
automorphisms is not strongly isotropy irreducible but it can be
isotropy irreducible (this happens if and only if $F$ is
irreducible in $T_oM$). The circle group $\bbT=\bbR/2\pi\bbZ$ is
also contained in this class.

We say that a function $u$ on $M$ is a polynomial if the linear
span of its translates  $u\circ g$, $g\in G$, is finite
dimensional. The polynomials are real analytic functions on $M$ {
since they can be lifted onto $G$ as matrix elements of finite
dimensional representations}. Let $\cE$ be a finite dimensional
$G$-invariant linear subspace of $L^2(M)$ and $\cS$ be the unit
sphere in it. (In the notation $L^2(M)$, the probability invariant
measure on $M$ is assumed.) In this paper, we compute or estimate
the averages over $\cS$ of some metric quantities, such as the
Hausdorff measures of level sets and their intersections or
$L^p$-norms of polynomials $u\in\cS$.

The strongly isotropy irreducible homogeneous spaces were
classified first by O.V.~Manturov (\cite{Mant1}--\cite{Mant3}) in
1961, independently by J.A.~Wolf (\cite{Wo}) in 1968, and by
M.~Kr\"amer (\cite{Kr}) in 1975. Their structure
was clarified in papers by M.~Wang and W.~Ziller  (\cite{WZ}),
E.~Heintze and W.~Ziller (\cite{HZ}). This class of homogeneous
spaces is closely connected with the symmetric spaces.

Due to Schur's lemma, $M$  admits a unique up to a scaling factor
$G$-invariant Riemannian metric. It follows from the uniqueness
that it is a quotient of some bi-invariant metric on $G$. We fix
these metrics and denote by $\De$ and $\td\De$ the corresponding
Laplace--Beltrami operators on $M$ and $G$, respectively.

Let the Lie algebra $\frg$ of $G$ be realized by right invariant
vector fields on $G$. {The Lie algebra of their projections onto
$M$ is its homomorphic image; if the action is virtually
effective\label{actef}\footnote{The action is effective if its
kernel is trivial and virtually effective if it is finite. The
kernel of the action consists of those $g\in G$ which define the
identical transformation of $M$.}, then the projection is an
isomorphism}. We shall assume this. Thus, we may identify these
Lie algebras. If $\xi_1,\dots,\xi_l$ is an orthonormal base in
$\frg$, then $\td \De=\xi_1^2+\dots+\xi_l^2$, where $l=\dim G$,
and
\begin{eqnarray}\label{lapve}
\De=\xi_1^2+\dots+\xi_l^2,
\end{eqnarray}
where $\xi_j$ denotes a vector field on $G$ as well as its
projection onto $M$. {  Note that $\xi_j$ may vanish on $M$:
$\xi_j(p)=0$ if and only if $\xi\in\frh_p$, the stable subalgebra
for $p\in M$. Since the vector fields in $\frg$ generate one
parameter subgroups of $G$, we have $\xi\cE\subseteq\cE$ for any
$G$-invariant finite dimensional linear subspace $\cE$ of $L^2(M)$
and all $\xi\in\frg$. By (\ref{lapve}),
\begin{eqnarray}\label{de-in}
\De\cE\subseteq\cE.
\end{eqnarray}
If $M$ is not isotropy irreducible, then (\ref{de-in}) may be
false. (It is true if the metric on $M$ is a quotient of a
bi-invariant metric on $G$.) Also, it follows from (\ref{de-in})
that any polynomial on $M$ is a finite linear combination of
eigenfunctions of $\De$.}

Throughout the paper, we use the notation
\begin{eqnarray*}
m=\dim M.
\end{eqnarray*}
For a real function $u$ on $M$ and $t\in\bbR$, set
\begin{eqnarray}\label{deflu}
\begin{array}{rcl}
L_u^t&=&\{p\in M:\,u(p)=t\},\\
U_u^t&=&\{p\in M:\,u(p)\geq t\}.
\end{array}
\end{eqnarray}
If $u$ is an eigenfunction of $\De$, then $L^0_u$ is called {\it a
nodal set}; we shall also use the notation $N_u$ for it. The
Hausdorff measure of dimension $k$ is denoted as $\frh^k$.  We
assume $t$ fixed and consider Hausdorff measures of these sets as
functions of $u$, and, for example,
$\frh^{m}(U^{t_1}_{u_1}\cap\ldots\cap U^{t_k}_{u_k})$ as functions
of $u_1,\ldots,u_k$. Fixing a probability measure on $\cE$ and
$t\in\bbR$, we get random variables
$\frh^{m-1}\left(L^t_u\right)$, $\frh^{m}\left(U^t_u\right)$,
etc.. Their distributions contain essential information on the
polynomials.

\subsection{Brief history}

To the best of my knowledge, investigations in this direction were
initiated by papers \cite{BP} by Bloch and Polya, \cite{PWZ32} by
Paley, Wiener, and Zygmund, and  \cite{LO38}, \cite{LO39} by
Littlewood and Offord. They considered the number of real zeroes
of algebraic equations with various types of random coefficients.
There are many papers in this area now; we describe briefly only
the results which are close to this article.

For real zeroes of random polynomials of one variable M.\,Kac in
\cite{Kac} proved an exact integral formula for the expectation
and found its asymptotic. Edelman and Kostlan in the paper
\cite{EK} noted that the expectation may be treated as the length
of some curve in a sphere due to a Crofton type formula in
spheres. They used this approach in some other situations.

For the Laplace-Beltrami eigenfunctions on compact manifolds,
Berard found the asymptotic of expectations of
$\frh^{m-1}\left(N_u\right)$, where $u$ runs
over the linear span of 
eigenfunctions corresponding to the eigenvalues of $-\De$ which
are less than $\la$, as $\la\to\infty$ (see \cite{Ber} for more
details).

The case of $\bbT$ is classical. We mention only papers \cite{Qu},
\cite{BBL}, \cite{GW}. For the trigonometric polynomials
\begin{eqnarray*}
u=\frac1{\sqrt{n}}\sum_{k=1}^n(a_k\cos kt+b_k\sin kt),
\end{eqnarray*}
where $a_k,b_k$ are Gaussian standard (i.e., with zero mean and
variance 1) random coefficients,  Qualls (\cite{Qu}) found the
expectations $E_n$ of the number of zeroes $Z_n(u)=\frh^0(N_u)$:
\begin{eqnarray}\label{qualls}
E_n=2\sqrt{\frac1n\sum\nolimits_{k=1}^nk^2}\sim\frac{2}{\sqrt3}n.
\end{eqnarray}
Bogomolny, Bohigas and Leboeuf conjectured in \cite{BBL} that
$\var Z_n=cn$ for some $c>0$. {  Granville and Wigman proved this
equality and, moreover, that $\frac{Z_n(u)-E_n}{\sqrt{cn}}$
converges weakly to the standard Gaussian distribution
(\cite{GW}).  The constant $c$ is equal to some complicated
explicitly written definite integral.}

In \cite{ORW}, Oravecz, Rudnick, and Wigman considered the
standard tori $\bbR^m/\bbZ^m$, a suitably normalized Gaussian
measure in the space $\cE_\la$ of $\la$-eigenfunctions, and the
Leray measure of a nodal set
\begin{eqnarray}\label{dlera}
\frl\left(N_u\right)=\lim_{\ep\to0}\frac1{2\ep}
\frh^m\left(U^{-\ep}_u\setminus U^{\ep}_u\right).
\end{eqnarray}
Note that the space $\cE_\la$ of all $\la$-eigenfunctions on
$\bbR^m/\bbZ^m$ is always invariant with respect to the
permutations of the coordinates and changes of their signs, in
other words, it is an invariant subspace of the semidirect product
$G$ of the torus $\bbR^n/\bbZ^n$ and the finite irreducible group
$\sfBC_n$, which is described above.
(In fact, the results of \cite{ORW} were obtained for
$G$-invariant subspaces of $\cE_\la$.) They calculated the
expectation of $\frl\left(N_u\right)$, which appears to be equal
to $\frac1{\sqrt{2\pi}}$ independently of $m$ and $\la$, and
proved for $m=2$ and $m\geq5$ that
\begin{eqnarray*}
\var\frl\left(N_u\right)\sim\frac1{4\pi\dim\cE_\la}
\end{eqnarray*}
as $\la\to\infty$. In \cite{RW}, Rudnick and Wigman proved that
the expectation of  $\frh^{m-1}\left(N_u\right)$ is asymptotic to
$C\sqrt{\la}$ and
\begin{eqnarray*}
\var\big(\la^{-\frac12}\frh^{m-1}
(N_u)\big)=O\big(\la^{-\frac12}\big)
\end{eqnarray*}
for the Gaussian distributions of $u\in\cE_\la$ on the tori
$\bbR^m/\bbZ^m$, $m\geq2$, assuming $\dim\cE_\la\to\infty$.

Let $S^m$ denote the unit sphere in $\bbR^{m+1}$. The $n$th
eigenspace $\cE_{\la_n}=\cH^m_n$ corresponds to the eigenvalue
$\la_n=n(n+m-1)$ and consists of traces of harmonic homogeneous of
degree $n$ polynomials on $S^m$.
This case 
was considered in the papers \cite{Ne}, \cite{Gi08}, \cite{W09},
\cite{W10}, \cite{MW1}, \cite{MW2}, where $u$ was subject either
to the Gaussian distribution in $\cE_\la$ or to the uniform one in
$\cS$.
\begin{remark}\rm
Both distributions mentioned above are rotation invariant. Since
$\frh^{m-1}\left(N_u\right)$ and $\frh^m\left(U_u^0\right)$ are
homogeneous of degree 0 on $u$, the resulting distributions of
$\frh^{m-1}\left(N_u\right)$ and $\frh^m\left(U_u^0\right)$ are
identical in the Gaussian and the uniform spherical cases. In
particular, the expectations and variances are equal. For the
level sets and for the Leray measures, this is not true but the
results for any of the two types of distributions can be deduced
from the results on the other one (for instance, we compute the
expectations for radial measures in Proposition~\ref{merad}). In
the papers cited above, except for the first and the second, the
authors work with the Gaussian distribution.
\end{remark}

Neuheisel proved that the normalized Hausdorff and Leray measures
on the nodal sets almost surely converges $*$-weakly to the
probability invariant measure as $\la\to\infty$ (for the precise
statement, see \cite{Ne}). He found the expectations of
$\frh^{m-1}\left(N_u\right)$ and $\frl\left(N_u\right)$ and
estimated their variances as
$O\Big(n^{-\frac{(m-1)^2}{3m+1}}\Big)$ and
$O\Big(n^{-\frac{m-1}{2}}\Big)$, $n\to\infty$, respectively. In
\cite{W09}, Wigman refined this: he proved that
$\var\left(\frl(N_u)\right)=\frac{c}{N}$, where $c$ depends only
on $m$, $N=\dim\cH_m^n$, and
$\var\left(\frh^{m-1}(N_u)\right)=O\left(\frac{\la}{\sqrt{N}}\right)$.
For $S^2$ he proved that $\var\left(\frh^1(N_u)\right)=c\ln
n+O(1)$ in \cite{W10} (there was an error in the calculation of
$c$ which had been corrected later). Marinucci and Wigman studied
the random area $\frh^2(S^2\setminus U^t_u)$. In \cite{MW1}, they
show that for a fixed $t\in\bbR$
\begin{eqnarray*}
\var\left(\frh^2(S^2\setminus
U^t_u)\right)=\frac{t^2\phi(t)}n+O\left(\frac{\log n}{n^2}\right)
\end{eqnarray*}
as $n\to\infty$, where $\phi$ is the standard Gaussian
distribution function. For $t=0$, it is proved in \cite{MW2} that
$\var\left(U^0_u\right)=\frac{C}{n^2}(1+o(1))$, where $n$ is even,
\begin{eqnarray*}
C=8\pi\int_0^\infty (\arcsin J_0(\tau)-J_0(\tau))\tau\,d\tau,
\end{eqnarray*}
$J_0$ is the Bessel function of the first kind and zero index; the
integral converges conditionally. (Actually, the authors
considered the difference between measures of sets of positivity
and negativity of $u$ (defect); their result differs by the factor
4 from $C$ above.)

In \cite{Gi08}, estimates for some metric quantities of the nodal
sets in spheres (in particular, the sharp upper bound for lengths
of the nodal sets of spherical harmonics on $S^2$) and the
expectations of Hausdorff measures of their intersections for the
uniform distributions were found, including the mean number of
common zeroes of $m$ independent random eigenfunctions on $S^m$.

Surveys \cite{JNT}, \cite{Ze08}, \cite{Ze12} describe the current
state of this area. They contain many known facts as well as
methods and open problems.

\subsection{Some observations and the results}
In this paper, we consider an arbitrary compact connected isotropy
irreducible homogeneous manifold $M$ and expectations of some
metric quantities for polynomials. { The results on variances,
higher moments, and estimates usually cannot be proved in such
generality; they will be considered in forthcoming papers.}
Throughout the text it is assumed that
\begin{itemize}
\item[($\sfE$)] $\cE$ is a finite dimensional $G$-invariant linear
subspace of $L^2(M)$ such that $\one\perp\cE$, where $\one(p)=1$
for all $p\in M$.
\end{itemize}
Also, $\cS$ everywhere denotes the unit sphere in $\cE$. Unless
the contrary is explicitly stated, the expectations relate to the
uniform distribution in $\cS$ (i.e., to the $\SO(\cE)$-invariant
probability measure).

We formulate below some useful observations.
\begin{enumerate}
\item Let $N$ be a Riemannian $G$-manifold and $\iota:\,M\to N$ be
an equivariant nonconstant smooth map. Then $\iota$ is a local
diffeomorphism onto its image since $M$ is isotropy irreducible
(hence $\ker d_p\iota=0$ for all $p\in M$). Therefore, $\iota$ is
a finite covering.

\item The restriction of the Riemannian metric in $N$ onto
$\iota(M)$ is proportional to the Riemannian metric in $M$ since
the invariant Riemannian metric in $M$ is unique up to a scaling
factor.

\item \label{obse1}Let $s$ be the coefficient of proportionality.
If $\ga$ is a path in $M$ of length $l$, then the path
$\iota\circ\ga$ has length $sl$. It follows that the inner
distances locally are also multiplied by $s$ (i.e., $\iota$ is a
local\footnote{ The two-sheeted expending covering $z\to z^2$ of
the unit circle $\bbT$ in $\bbC$ is a simple example of a
non-global local homothety} metric homothety) and the same is true
for the Hausdorff measure $\frh^k$, with the coefficient $s^k$.

\item There is a natural equivariant immersion $\iota:\,M\to\cS$.
For $p\in M$, let $\phi_p\in\cE$ be such that
$u(p)=\scal{\phi_p}{u}$ for all $u\in\cE$ and set
$\iota(p)=\frac{\phi_p}{|\phi_p|}$.

\item For the immersion $\iota$ we have
$s=\sqrt{\frac{|\Tr\De|}{\dim M\dim\cE}}$ (see Lemma~\ref{forms}
below). If $\cE$ is an eigenspace of $\De$, then
$s=\sqrt{\frac{\la}{m}}$, where $\la$ is the eigenvalue of $-\De$.

\item Using a Crofton type formula of Integral Geometry in spheres
(see Theorem~\ref{feder}), one can compute averages (expectations)
of $\frh^{k-1}(L_u^{t}\cap X))$, $\frh^{k}(U_u^{t}\cap X))$ for
subsets $X$ of $M$, and some other functions of $u$, with respect
to the probability invariant measure on the sphere $\cS$.
\end{enumerate}
{ They seem to be already known (may be, except for the third and
the fifth; the coefficient $s$ appeared implicitly in several
papers, for example, in \cite{Ne}). The scheme was realized in the
paper \cite{Gi08} for $M=S^m=\SO(m+1)/\SO(m)$, the unit sphere in
$\bbR^{m+1}$ and the spaces of spherical harmonics on them, which
can be characterized as eigenspaces of $\De$ or as irreducible
components of $L^2(S^m)$.} Clearly, $S^m$ is isotropy irreducible,
moreover, the stable subgroup $H=\SO(m)$ acts transitively on the
unit sphere in $T_oS^m$. It turns out that the assumption that
$\cE$ is an eigenspace of $\De$ is not essential for the
computation of expectations, the results depends (at most) on the
coefficient $s$. (However, for variances and upper or lower bounds
this is not true usually.)


We compute the expectations of Hausdorff measures of level sets
and their intersections (Theorem~\ref{expct}). In particular, for
$L^t_u$ and $U_u^t$ we get
\begin{eqnarray}
\sfM(\frh^{m-1}(L_u^{ct}))=\varpi\frac{\varpi_{m-1}}{\varpi_{m}}
s\left(1-t^2\right)^{\frac{d-1}2},\label{mealu}\\
\sfM(\frh^m(U_u^{ct}))=
\varpi\frac{\varpi_{d-1}}{\varpi_{d}}\int_t^1
\left(1-\tau^2\right)^{\frac{d}{2}-1}\,d\tau,\label{meauu}
\end{eqnarray}
where $m=\dim M$, $\cE,s$ are as above, $c^2=d+1=\dim\cE$,
$\varpi$, $\varpi_k$ are volumes of $M$, $S^k$, respectively,
$t\in\bbR$ is fixed, and
\begin{eqnarray*}
\sfM(f)=\int_\cS f(u)\,du
\end{eqnarray*}
for a function $f$ on $\cS$ ($du$ corresponds to the invariant
probability measure on $\cS$).

The formulas above hold for all isotropy irreducible homogeneous
spaces, in particular, for spheres $\SO(m+1)/\SO(m)$ and for the
standard tori $T_m=\bbR^m/\bbZ^m$ considered as a homogeneous
space of $T_m$ extended by the finite group $\sfBC_m$ of all
compositions of permutations and componentwise inversions in
$T_m$. This is equivalent to the assumption that the
$T_m$-invariant space $\cE$ (equivalently, its spectrum) is
$\sfBC_m$-invariant; in fact, it was assumed in the paper
\cite{ORW} (clearly, the spectrum of any eigenspace on $T_m$ is
$\sfBC_m$-invariant). Qualls's formula (\ref{qualls}) follows from
(\ref{mealu}) with $M=G=\bbT=\bbR/2\pi\bbZ$ and $t=0$ since
$\varpi_0=2$, $\varpi=\varpi_1=2\pi$, and
$s=\sqrt{\frac1n\sum_{k=1}^nk^2}$ according to  Lemma~\ref{forms}.
Moreover, (\ref{mealu}) can be applied to spaces $\cE$ with
arbitrary spectra, for example, to the space of trigonometric
polynomials of the type
\begin{eqnarray*}
\sum_{i=1}^n(a_{k_i}\cos k_it+b_{k_i}\sin k_it),
\end{eqnarray*}
where $0<k_1<\dots<k_n$ are integer. By (\ref{mealu}) and
Lemma~\ref{forms}, the expectation equals $2s$, where
\begin{eqnarray*}
s=\sqrt{\frac1n\sum\nolimits_{i=1}^nk_i^2}.
\end{eqnarray*}
A similar formula can be derived for the intersections of sets
$L^t_u$ and $U^t_u$, for the natural extension of the Leray
measure onto all level sets (see (\ref{defle}) for the
definition), and for quantities of the type $\int_Mf(u(p))\,dp$.
The results are stated in Theorem~\ref{expct}. For $f(t)=|t|^a$ we
derive the explicit formula (\ref{taexe}) (Theorem~\ref{exmua}),
which holds for all $a>-1$; its right-hand side is independent of
$M$. Setting $a=1$, we get the expectation of $L^1$-norm:
\begin{eqnarray*}
\sfM\left(\|u\|_1\right)=\sqrt{\frac{d+1}{\pi}}\,
\frac{\Ga\left(\frac{d+1}{2}\right)}
{\Ga\left(\frac{d+2}{2}\right)}
\end{eqnarray*}
where $d=\dim\cS=\dim\cE-1$. The right-hand side decreases with
$d$.
If $d=1$, then it is equal to $\frac{2\sqrt2}{\pi}\approx0.9$~
and it tends to $\sqrt{\frac2\pi}\approx0.8$~
as $d\to\infty$. For any $a>-1$ we have
\begin{eqnarray*}
E(a,d):=\sfM\left(\int_M|u(p)|^a\,dp\right)\to
~2^{\frac{a}{2}}\frac{\Ga\left(\frac{a+1}{2}\right)}{\sqrt\pi},
\end{eqnarray*}
as $d\to\infty$, where $E(a,d)$ increases with $d$ if $a>2$ or
$a\in(-1,0)$ and decreases if $a\in(0,2)$ (for $a=0$ and $a=2$ the
equality holds). As $a\to-1$,
$E(a,d)\sim\frac{A}{a+1}$,
where $A$ depends only on $d$. The integral $\int_M|u(p)|^a\,dp$
may diverge for arbitrary small negative $a$, for example, if
$M=\bbT$ and $u(t)=(\sin t)^{2k+1}$, but the averages are finite
if $a>-1$.

The computation of $E(a,d)$ makes it possible to estimate from
above the expectation of norms $\|u\|_p$ in $L^p(M)$, $1\leq
p<\infty$ (Theorem~\ref{lpinf}). If $p\geq2$, then
\begin{eqnarray}\label{boum}
\sfM\left(\|u\|_p\right)<\sqrt{\frac{p+1}e}.
\end{eqnarray}
For $p\in[1,2)$ the same inequality holds if $\dim\cE$ is
sufficiently large. Note that the bound is independent of $\cE$
and $M$. Inequalities like $\|u\|_p\leq C\sqrt{p}\,\|u\|_2$ are
known for the trigonometric lacunary series. Perhaps, this upper
bound cannot be improved essentially.

The uniform estimate for the expectations does not yield a similar
estimate for individual eigenfunctions. For example, if
$M=S^2\subset\bbR^3$, then for any $p>2$ and spherical harmonics
$\vf_n(x,y,z)=c\Re(x+yi)^n$ of degree $n$ such that
$\|\vf_n\|_2=1$ we have $\|\vf_n\|_p\to\infty$ as $n\to\infty$
(see \cite{So88}). On the other hand, if $M=\bbR^2/\bbZ^2$, then
$\|u\|_4\leq C\|u\|_2$ for all eigenfunctions $u$ by a result of
Zygmund (\cite{Zy}). Similar problems for $L^4$ norms of elements
of a random orthogonal base in the space $\cH_n$ of spherical
harmonics on $S^2$ of degree $n$ were considered in \cite{SZ10};
in particular, this paper contains a sketch of the proof of an
estimate for the average of the functional
$\sum_{j=1}^{2n+1}\|u_k\|_4^4$ of such a base.

Estimates for $\sfM\left(\|u\|_\infty\right)$ cannot be derived
from the results on $\sfM\left(\|u\|_p\right)$ directly while for
any fixed $u\in\cE$ we have $\|u\|_p\to\|u\|_\infty$ as
$p\to\infty$.  Indeed, the upper bound (\ref{boum}) for
$\sfM\left(\|u\|_p\right)$ is independent of $\dim\cE$ but  it may
happen that $\sfM\left(\|u\|_\infty\right)\to\infty$ as
$\dim\cE\to\infty$, for example, this is true if $\cE$ is
contained in a subspace of $L^2(\bbT)$ with a lacunary spectrum.
There is the evident sharp upper bound $\sqrt{\dim\cE}=:c$ for
$\|u\|_\infty$ (see (\ref{leqc})) which is attained on
$u=\iota(q)\in\cS$ for any $q\in M$. We get the estimate
\begin{eqnarray}\label{ineli}
\sfM\left(\|u\|_\infty\right)<(e^{m-\frac12}+\ep)\sqrt{\ln \ka},\\
\ka=cs,\nonumber
\end{eqnarray}
which holds for any $\ep>0$ and sufficiently large $\ka$. (In
fact, $\ka$ is the norm of the identity operator $\cE\to\cE$ with
norms of $L^2(M)$ and $\Lip(M)$, respectively; see
Lemma~\ref{lipsc}.) Similar upper bounds for the random variable
$\|u\|_\infty$ appeared in various problems. For instance, the
analogous inequalities of weak type for $\|u\|_\infty$ are
contained in Kahane's book \cite[Ch. 6]{Ka}; in \cite{Ne},
Neuheisel proved for spherical harmonics on $S^m$ that
$\|u_n\|_\infty=O(\sqrt{\ln n})$ almost surely as $n\to\infty$,
where $u_n\in\cS_n\subseteq \cE_n$ and $\cE_n$ is the space of
spherical harmonics of degree $n$. In these cases, the bounds
$\sqrt{\ln\ka}$, $\sqrt{\ln s}$, $\sqrt{\ln c}$, and $\sqrt{\ln
n}$ are equivalent. Shiffman and Zelditch in the paper \cite{SZ03}
considered spaces $\cH_n=H(M,\cL^n)$ of holomorphic sections of
the $n$th power of a positive line bundle $\cL$ over a compact
K\"ahler manifold $M$. These spaces are treated as subspaces of
$L^2(X)$, where $X$ is the $S^1$ bundle associated to $\cL$. (If
$M$ is the projective space $\bbC\bbP^m$ and $\cL$ is the natural
line bundle over it, then $\cH_n$ is the space of holomorphic
homogeneous polynomials of degree $n$ on $\bbC^{m+1}$ restricted
to the unit sphere $X=S^{2m+1}$.) Such a construction can be
applied to the almost complex manifolds. In particular, they
proved that $L^p$-norms of sequences of elements of the unit
spheres $\cS\cH_n$ in $\cH_n$ are $O(1)$ and $O(\sqrt{\log n})$
almost surely as $n\to\infty$ for $p<\infty$ and $p=\infty$,
respectively. The key ingredients of their methods, which can work
in the real setting as well, are the asymptotics of reproducing
kernels and the concentration of measure estimates (see, e.g.,
\cite{Le01}). The proof of (\ref{ineli}) in this paper is based on
Lemma~\ref{lipsc} and the estimate (\ref{boum}).

The inequality (\ref{ineli}) contains  no information if
$K^2\ln\ka>\dim\cE$. This happens, for example, for the spaces of
trigonometric polynomials with the spectrum $\{1,2,\dots,n,n!\}$
if $n$ is sufficiently large. Thus, (\ref{ineli}) should be
refined. Probably, the right-hand side of (\ref{ineli}) exhibits
the sharp order of growth for strictly isotropy irreducible
homogeneous spaces if $\cE$ is an eigenspace.

\section{Preparatory material}
In what follows, we keep the notation of the introduction; $|~|$
and $\scal{~}{~}$ denote the norm and inner product in Euclidean
spaces, respectively, $\|\ \|_p$ is the norm in $L^p(M)$. We write
$\int_M f(p)\,dp$, $\int_G f(g)\,dg$, etc., for the integration
over invariant probability measures on $M$, $G$, and other
homogeneous spaces of  compact groups; also, these measures are
assumed in the notation $L^2(M)$, $L^2(G)$, etc.. Thus, we
consider on $M$ two finite invariant measures: the probability one
and $\frh^m$, where $m={\dim M}$. Set
\begin{eqnarray*}
\varpi=\frh^m(M);
\end{eqnarray*}
then $d\frh^m=\varpi\, dp$. Functions are assumed to be
real-valued unless the contrary is explicitly stated. The space
$\cE$ is always assumed to satisfy ($\sfE$). Since we consider
real functions and $M$ is compact, ($\sfE$) implies
\begin{eqnarray*}
\dim\cE>1.
\end{eqnarray*}
Being finite dimensional and $G$-invariant, $\cE$  consists of
real analytic functions. For any $p\in M$ there exists the unique
$\phi_p\in\cE$ which realizes the evaluation functional at $p$:
\begin{eqnarray*}
\scal{u}{\phi_p}=u(p)
\end{eqnarray*}
for all $u\in\cE$. Set
\begin{eqnarray*}
\phi(p,q)=\scal{\phi_p}{\phi_q},\quad p,q\in M.
\end{eqnarray*}
Then $\phi(p,q)=\phi_p(q)=\phi(q,p)$; moreover, $\phi(x,y)$ is the
{reproducing kernel for $\cE$} (i.e., the mapping
$u(x)\to\int_M\phi(x,y)u(y)\,dy$ is the orthogonal projection onto
$\cE$ in $L^2(M)$). Due to the homogeneity of $M$,
$|\phi_p|=|\phi_o|\neq0$ for all $p\in M$. Since the trace of the
projection is equal to $\int_M\phi(x,x)\,dx$, we have
\begin{eqnarray*}
\phi(o,o)=|\phi_o|^2=\dim\cE.
\end{eqnarray*}

{ By the first observation on Page~\pageref{obse1}, the
equivariant mapping $\iota:\,M\to\cS$,
\begin{eqnarray}
\iota(p)=\frac{\phi_p}{|\phi_p|},\label{iotae}
\end{eqnarray}
of the forth observation is an immersion. The denominator in
(\ref{iotae}) is independent of $p$ since $M$ is homogeneous.
Clearly, $\iota$ is an embedding if and only if $\cE$ separates
points of $M$.

The scaling factor $s$ of the second observation admits an evident
expression on the level of tangent spaces:
\begin{eqnarray}
s=s(\cE)=\frac{|d_p\iota(v)|}{|v|},\label{defsn}
\end{eqnarray}
where the right-hand side is independent of $p\in M$ and $v\in T_p
M\setminus \{0\}$. According to the third observation, the length
of curves and (locally) the inner distance in $\iota(M)$ defined
by the Riemannian metric in $\cS$ (equivalently, by the Euclidean
metric in $\cE$) are proportional to that of $M$.}

For short, we shall denote
\begin{eqnarray*}
\iota(p)&=&\bar p,\\
\iota(M)&=&\bar M.
\end{eqnarray*}
The following notation will be used throughout the paper:
\begin{eqnarray*}
d&=&\dim\cE-1=\dim\cS,\\
c&=&|\phi_o|=\sqrt{\phi(o,o)}=\sqrt{d+1},\\
\cS^t_u&=&\{x\in\cS:\,\scal{x}{u}=t\},\\
\cU^t_u&=&\{x\in\cS:\,\scal{x}{u}\geq t\},
\end{eqnarray*}
where $u\in\cE$ and $t\in\bbR$. Clearly, for all $u\in\cS$, $p\in
M$
\begin{eqnarray}\label{leqc}
|u(p)|\leq c,
\end{eqnarray}
where the equality holds only for $u=\bar p$. If $t\in[-1,1]$,
then we obviously have
\begin{eqnarray}
\iota\left(L^{ct}_u\right)=\cS^{t}_u\mathop\cap\bar M,\label{embha}\\
\iota\left(U^{ct}_u\right)=\cU_u^t\mathop\cap\bar M.\label{embhu}
\end{eqnarray}

A set which can be realized as a Lipschitz image of a bounded
subset of $\bbR^k$ is called {\it $k$-rectifiable} (we consider
only countable unions of compact sets). If $\iota$ is one-to-one
on an $r$-rectifiable set $X\subseteq M$, then, due to
(\ref{defsn}),
\begin{eqnarray}\label{iodil}
\frh^r_\cS(\iota(X))=s^{r}\frh^r_M(X).
\end{eqnarray}
In this equality and in the sequel, the Hausdorff measures
correspond to the metrics in the related spaces. We shall drop the
lower index usually. Let $S^k$ be the unit sphere in $\bbR^{k+1}$.
Set
\begin{eqnarray*}
\varpi_k=\frh^k(S^k)=\frac{2\pi^{\frac{k+1}2}}
{\Ga\left(\frac{k+1}{2}\right)}, \\
\vk_d(t)=\frh^{d}\left(\cU^t_u\right)=
\varpi_{d-1}\int_t^1\left(1-\tau^2\right)^{\frac{d}{2}-1}\,d\tau,
\end{eqnarray*}
where $-1\leq t\leq 1$. Note that the second term is independent
of $u\in\cS$ and that $\vk_d(-1)=\varpi_d$,
$\vk_d(0)=\frac{\varpi_d}2$. Also, we assume that
$\vk(t)=\varpi_d$ if $t\leq-1$ and $\vk(t)=0$ for $t\geq1$. The
definition (\ref{dlera}) may be extended onto all level sets:
\begin{eqnarray}\label{defle}
\frl(L_u^t)=\limsup_{\ep\to+0}\frac1{2\ep}\frh^m\left(U^{t-\ep}_u\setminus
U^{t+\ep}_u\right).
\end{eqnarray}
{ We allow $\frl(L_u^t)$ to take the value $+\infty$ and assume
that $\frl(L_u^t)=0$ if $L_u^t=\varnothing$. Thus, (\ref{defle})
defines $\frl(L^t_u)$ for all $t\in\bbR$ and $u\in\cE$.  The
quantities $\frh^{m-1}(L^t_u)$, $\frh^m\left(U^t_u\right)$, etc.,
can be extended onto $\bbR\times\cE$ similarly.} Since
$\frh^m\left(U^t_u\right))$ is piecewise real analytic (see the
beginning of the next section), $\frl(L^t_u)$ is real analytic
outside a finite subset of $\bbR$. If $t=0$, then $\frl(L_u^t)$ is
called the {\it Leray measure}
of the {nodal set} $N_u=L^0_u$.  Since
$\frh^m\left(U_u^\tau\right)$ is non-increasing on $\tau$, for
almost all $t\in\bbR$
\begin{eqnarray}\label{derme}
\frl(L_u^t)=-\frac{d}{d\tau}\frh^m\left(U_u^\tau\right)\big|_{\tau=t},
\end{eqnarray}
where the derivative exists almost everywhere. Due to the coarea
formula { (see \cite[Theorem~3.2.12]{Fe})},
\begin{eqnarray*}
\frh^m(U_u^t)=\int_t^\infty\Big(\int_{L^\tau_u}
\frac{d\frh^{m-1}(p)}{|\nabla u(p)|}\Big)\,d\tau
\end{eqnarray*}
and almost everywhere on $[-c,c]$  we have
\begin{eqnarray}\label{deler}
\frl(L_u^t)=\int_{L_u^t}\frac{d\frh^{m-1}(p)}{|\nabla u(p)|}.
\end{eqnarray}

The following theorem is a simplified version of Theorem~3.2.48 in
\cite{Fe}.
\begin{theorem}\label{feder}
Let $A,B\subseteq S^{d}$ be compact, $A$ be $k$-rectifiable, $B$
be $j$-rectifiable, and $\vf,\psi$ be continuous functions on $A$,
$B$, respectively. Set $r=k+j-d$. Suppose $r\geq0$. Then
\begin{eqnarray*}
\int_{\OO(d+1)}\int_{A\cap gB}\vf(x)
\psi(g^{-1}x)\,d\frh^r(x)\,dg=K\int_A\vf(x)\,d\frh^k(x)\,\int_B\psi(x)\,d\frh^j(x),
\end{eqnarray*}
where $K=\frac{\Ga\left(\frac{k+1}{2}\right)
\Ga\left(\frac{j+1}{2}\right)}{2\Ga\left(\frac{1}{2}\right)^d
\Ga\left(\frac{r+1}{2}\right)}={\displaystyle
\frac{\varpi_r}{\varpi_k\varpi_j}}$. \qed
\end{theorem}
In particular, for $\vf=\psi=1$ we have
\begin{eqnarray}\label{intca}
\int_{\OO(d+1)}\frh^r(A\cap gB)\,dg=K\frh^k(A)\frh^j(B).
\end{eqnarray}

Let us fix an orthogonal decomposition of $\cE$ into a sum of
$G$-invariant subspaces:
\begin{eqnarray}\label{decel}
\cE=\sum_{j=1}^l\oplus\,\cE^{j},
\end{eqnarray}
where $\De u=-\la_j u$ for $u\in\cE^j$. We do not assume that
$\la_j\neq\la_k$ if $j\neq k$. It follows from (\ref{decel}) that
\begin{eqnarray}\label{decphi}
\phi_p=\sum_{j=1}^l\phi_p^{j},
\end{eqnarray}
where $\phi_p^{j}\in\cE^{j}$ represents the evaluation functional
at $p\in M$ on $\cE^{j}$, $j=1,\dots,l$.

Let each of $\cE_1,\dots,\cE_l$ be as $\cE$ above; for
$i\in\{1,\dots,l\}$, we equip the notations of the objects related
to $\cE_i$ with the lower index $i$.
Further, we write ${\bft}=(t_1,\dots,t_l)\in\bbR^l$,
\begin{eqnarray}\label{boldnot}
\begin{array}{rclccrcl}
\bfE&=&\cE_1\times\dots\times\cE_l,&&
{\bfS}&=&\cS_1\times\dots\times\cS_l,\\
\bfu&=&(u_1,\dots,u_l)\in\bfE,&& d{\bfu}&=&du_1\dots du_l,\\
L^{\bf t}_{\bfu}&=&L^{t_1}_{u_1}\mathop\cap\dots\mathop\cap
L^{t_l}_{u_l},&&
U^{\bft}_{\bfu}&=&U^{t_1}_{u_1}\mathop\cap\dots\mathop\cap
U^{t_l}_{u_l},
\end{array}
\end{eqnarray}
and so on.  The mean value (expectation) of a function $f$ on
$\bfS$ or $\cS$ is denoted as $\sfM(f)$:
\begin{eqnarray*}
\sfM(f)=\int_\bfS f(\bfu)\,d\bfu,\quad \sfM(f)=\int_\cS f(u)\,du.
\end{eqnarray*}

Let $\frg,\frh$ be the Lie algebras of $G,H$ respectively,
realized\footnote{  we assume that the action of $G$ is virtually
effective; see the footnote on page~\pageref{actef}} as Lie
algebras of vector fields on $M=G/H$ and let $G$ be equipped with
a bi-invariant Riemannian metric such that the projection
$\eta:\,G\to G/H=M$ is a metric submersion. Then $d\eta:\,\frg\to
T_oM$ is isometric on $\frh^\bot$, $\ker d\eta=\frh$, and $\De$
may be defined by (\ref{lapve}).

In the following two lemmas, we perform some calculations of
\cite{Gi08} in a more general setting. First, we find the
coefficient $s=s(\cE)$ of the local metric homothety
$\iota:\,M\to\bar M$ for $\cE$ in (\ref{decel}) (the case of
$M=S^m$ and one eigenspace was considered in \cite{Gi08}). Set
\begin{eqnarray*}
d_j=\dim\cE^j-1,\quad
\al_j=\frac{d_j+1}{d+1}=\frac{\dim\cE^j}{\dim\cE}=\frac{c_j^2}{c^2},
\end{eqnarray*}
$j=1,\dots,l$. Thus, $\sum_{j=1}^l\al_j=1$.
\begin{lemma}\label{forms}
Let $s_j=s(\cE^j)$, $j=1,\dots,l$, and $\al_j$ be as above. Then
\begin{eqnarray}\label{emcoe}
s^2=\al_1s_1^2+\dots\al_ls_l^2=\frac{|\Tr\De|}{m\dim\cE},
\end{eqnarray}
where $\Tr\De$ is the trace of $\De$ in $\cE$.
\end{lemma}
\begin{proof}
Since $\iota$ is equivariant, for all $\xi\in\frg$
\begin{eqnarray}\label{difio}
d_o\iota(\xi(o))=\frac1{c}\xi\phi_o.
\end{eqnarray}
We may choose the basis in (\ref{lapve}) such that
$\xi_{m+1},\dots,\xi_k\in\frh$, where $k=\dim\frg$. Then remaining
$\xi_j$ are orthogonal to $\frh$. Since $\eta:\,G\to G/H=M$ is a
metric submersion, we have $|\xi_i(o)|=1$ for $i=1,\dots,m$;
clearly, $\xi_i(o)=0$ for $i=m+1,\dots,k$. Using consequently
(\ref{defsn}), (\ref{difio}), (\ref{iotae}), (\ref{lapve}), and
(\ref{decphi}), we get
\begin{eqnarray*}
ms^2=s^2\sum_{i=1}^k|\xi_i(o)|^2=\sum_{i=1}^k|d_o\iota(\xi_i(o))|^2=
\frac1{c^2}\sum_{i=1}^k|\xi_i\phi_o|^2=
-\frac1{c^2}\sum_{i=1}^{k}\scal{\xi_i^2\phi_o}{\phi_o}\\
=-\frac1{c^2}\scal{\De\phi_o}{\phi_o}=\frac1{c^2}\scal{\sum_{j=1}^l\la_j\phi_o^j}{\phi_o}=
\frac1{c^2}\sum_{j=1}^l\la_j|\phi_o^j|^2=\sum_{j=1}^l\al_j\la_j.
\end{eqnarray*}
If $l=1$, then $s=\sqrt{\frac{\la_1}{m}}$; hence,
$s_j=\sqrt{\frac{\la_j}{m}}$. This proves the first equality in
(\ref{emcoe}); the second is true since
$|\phi_o^j|^2=\dim\cE^j=c_j^2$, $c^2=\dim\cE$.
\end{proof}

The second lemma is similar to Lemma~6 of the paper \cite{Gi08},
where (\ref{downa})  was proved for $t=0$.
\begin{lemma}\label{rectx}
Let $|t|\leq1$ and $X\subseteq M$ be $(r+1)$-rectifiable, where
$r\leq m-1$. Then
\begin{eqnarray}\label{downa}
\int_{\cS}\frh^r(L_u^{ct}\mathop\cap X)\,du=
\frac{\varpi_r}{\varpi_{r+1}}s\frh^{r+1}(X)
\left(1-t^2\right)^{\frac{d-1}2}, \\
\int_{\cS}\frh^{r+1}(U_u^{ct}\mathop\cap X)\,du=
\frh^{r+1}(X)\frac{\vk_d(t)}{\varpi_{d}}.\label{ununa}
\end{eqnarray}
\end{lemma}
\begin{proof}
Since both sides of (\ref{downa}) are additive on $X$ and $\iota$
is a finite covering, it is sufficient to prove (\ref{downa})
assuming that $\iota$ is injective on $X$.  For each $t\in[-1,1]$,
the group $O(\cE)$ acts transitively on the family of spheres
$\{\cS_u^t\}_{u\in\cS}$. Due to (\ref{embha}), we may apply
Theorem~\ref{feder} to $\cS$ setting $A=\cS^t_o$, $k=d-1$,
$B=\iota(X)$, $j=r+1$ in (\ref{intca}). Since the Euclidean radius
of $\cS^t_{\bar o}$ is equal to $\sqrt{1-t^2}$, we have
\begin{eqnarray*}
\frh^{d-1}\left(\cS^t_{\bar o}\right)=\varpi_{d-1}
\left(1-t^2\right)^{\frac{d-1}2}.
\end{eqnarray*}
Using (\ref{embha}), (\ref{iodil}), replacing integration over
$\cS$ with averaging over $\OO(\cE)$, and applying (\ref{intca}),
we get (\ref{downa}):
\begin{eqnarray*}
\int_{\cS}\frh^r(L_u^{ct}\mathop\cap X)\,du=
\frac1{s^r}\int_{\cS}\frh^r(\iota(L_u^{ct}\cap X))\,du
=\frac1{s^r}\int_{\cS}\frh^r(\cS^t_u\cap\iota(X))\,du\\
=\frac1{s^r}\int_{\OO(\cE)}\frh^r(g\cS_{\bar
o}^t\cap\iota(X))\,dg=
\frac1{s^r} K\frh^{d-1}\left({\cS_{\bar o}^t}\right)\frh^{r+1}(\iota(X))\\
=\frac{\varpi_r}{s^r\varpi_{r+1}}\left(1-t^2\right)^{\frac{d-1}2}
\frh^{r+1}(\iota(X))= s\left(1-t^2\right)^{\frac{d-1}2}
\frac{\varpi_r}{\varpi_{r+1}}\frh^{r+1}(X).
\end{eqnarray*}
To prove (\ref{ununa}), set $A=\cU^t_{\bar o}$, $k=d$,
$B=\iota(X)$,
$j=r+1$ in (\ref{intca}). Then 
\begin{eqnarray*}
\int_{\cS}\frh^{r+1}(U_u^{ct}\mathop\cap X)\,du
=\frac1{s^{r+1}}\!\!\int_{\OO(\cE)}\!\frh^{r+1}(g\cU^t_{\bar o}\cap\iota(X))\,dg\\
=\ \frac1{s^{r+1}}\frac{\varpi_{r+1}}{\varpi_{d}\varpi_{r+1}}
\vk_{d}(t)\frh^{r+1}(\iota(X))
=\frac{\vk_d(t)}{\varpi_{d}}\frh^{r+1}(X),
\end{eqnarray*}
where we omit some steps since they are similar to those above.
\end{proof}

\section{Computation of expectations}

We formulate below two \L{}ojasiewicz's inequalities following
\cite{BM88} but in a weaker form (cf. Theorem~6.4, Remark~6.5, and
Proposition~6.8 of \cite{BM88}). Let $\CV_N(f)$ denote the set of
critical values of a smooth function $f$ on a manifold $N$ (we
drop the index if no confusion can occur). Suppose $f$ real
analytic in a domain $\cD\subseteq\bbR^n$, $N_f=f^{-1}(0)$. Let
$Q$ be a compact subset of $\cD$. Then there exist $\nu,\eta>0$
such that
\begin{eqnarray}\label{loja1}
|f(q)|\geq\eta\dist(q,N_f)^\nu
\end{eqnarray}
for all $q\in Q$, where $\dist$ denotes the Euclidean distance.
Furthermore, for any $x\in N_f$ there are its neighborhood $U$ and
$\eta>0$, $\theta\in(0,1)$ such that for all $q\in U$
\begin{eqnarray}\label{loja2}
|\nabla f(q)|\geq \eta|f(q)|^\theta,
\end{eqnarray}
where $\nabla$ stands for the Euclidean gradient.

\begin{lemma}\label{lecon}
For all $u\in\cS$
\begin{itemize}
\item[\rm(\romannumeral1)] the set $\CV(u)$ is finite,

\item[\rm(\romannumeral2)] $\frh^{m-1}(L^t_u)$ and $\frl(L^t_u)$
are smooth on $u\left(M\right)\setminus\CV(u)$ as functions of{\,}
$t$,

\item[\rm(\romannumeral3)] for any $t_0\in\CV(u)$ there are
$\theta\in\left(0,1\right)$ and $\eta>0$ such that
\begin{eqnarray}\label{lerbo}
\frl(L^t_u)\leq \eta\frh^{m-1}\left(L^t_u\right)|t-t_0|^{-\theta},
\end{eqnarray}
where $t\in u(M)\setminus\CV(u)$.
\end{itemize}
\end{lemma}
\begin{proof}
({\romannumeral1}). The set of critical points of $u$ is defined
by the equation $|\nabla u(x)|^2=0$. Hence, it has a finite number
of components being an analytic set in a compact manifold
(moreover, $\iota$ maps this set onto a real algebraic set in
$\cE$ since $\bar M$ can be distinguished in $\cE$ by
$G$-invariant polynomials). On the other hand, the set $\CV(u)$
has zero Lebesgue measure in $\bbR$ by Sard's theorem since $u$ is
sufficiently smooth. Hence $u$ is constant on each component.

({\romannumeral2}).  By ({\romannumeral1}), the set
$u(M)\setminus\CV(u)$ is the union of a finite family of disjoint
open intervals. Let $I$ be a compact subinterval in
$u(M)\setminus\CV(u)$ and let $t\in I$. By a basic theorem of
Morse Theory, $u^{-1}(I)$ is diffeomorphic to $I\times N$, where
$N=u^{-1}(t)$ is a smooth submanifold of $M$ (see
\cite[Theorem~3.1]{Mi63} or \cite[9.3.3]{PT88}). Together with the
coarea formula and (\ref{deler}), this implies ({\romannumeral2}).

({\romannumeral3}). Set $f(p)=u(p)-t_0$. Every point in the
critical level set $L^{t_0}_u$ has a neighborhood in $M$ where
(\ref{loja2}) holds. Standard compactness arguments and
({\romannumeral1}) show that (\ref{loja2}) is true in some
neighborhood of $L^{t_0}_u$ in $M$ (we may assume that $|f(p)|<1$
in every neighborhood, then we may increase $\theta$ keeping the
inequality (\ref{loja2}) and the inclusion $\theta\in (0,1)$).
Applying (\ref{deler}) with $t=u(p)$, we get (\ref{lerbo}) in some
neighborhood of $t_0$; it admits an extension onto $u(M)$ with,
may be, a smaller $\eta$.
\end{proof}

In the following theorem, we use the notation of (\ref{boldnot}).
For a function $f$ on $[-c,c]$ set
\begin{eqnarray}\label{defif}
I_f(u)=\int_{M}f(u(p))\,dp.
\end{eqnarray}
\begin{theorem}\label{expct}
Let  all factors in $\bfE$ satisfy {\rm($\sfE$)}, $X\subseteq M$
be $r$-rectifiable for some integer $r$, $1\leq r\leq m$,
$l\in\bbN$, $\bft=(t_1,\dots,t_l)$, and $t_i\in[-c_i, c_i]$ for
all $i=1,\dots,l$.
\begin{itemize}
\item[\rm{({1})}] If~ $l\leq r$, then
\begin{eqnarray}
\sfM\left(\frh^{r-l}\left(L^{\bf t}_{\bf u}\cap X\right)\right)=
\frac{\varpi_{r-l}}{\varpi_{r}}\frh^r(X)
\prod_{i=1}^ls_i\left(1-\frac{t_i^2}{c_i^2}\right)^{\frac{d_i-1}2},
\label{meanlh}
\end{eqnarray}
where $s_i$ is subject to {\rm(\ref{emcoe})} with $\cE=\cE_i$,
$i=1,\dots,l$.

\item[\rm{({2})}] For any $l\in\bbN$
\begin{eqnarray}\label{meanuh}
\sfM\left(\frh^{r}\left(U^{\bf t}_{\bf u}\cap X\right)\right)=
\frh^r(X)\prod_{i=1}^l\frac{\vk_{d_i}
\left(\frac{t_i}{c_i}\right)}{{\varpi_{d_i}}}.
\end{eqnarray}

 \item[\rm{({3})}] For almost all $t\in[-c,c]$
\begin{eqnarray}\label{lmean}
\sfM\left(\frl(L_u^t)\right)
=\frac{\varpi\varpi_{d-1}}{c\varpi_d}\left(1-\frac{t^2}{c^2}\right)^{\frac
d2-1}.
\end{eqnarray}
Moreover, if $f$ is a piecewise continuous function on $[-c,c]$,
then
\begin{eqnarray}\label{fmean}
\sfM\left(I_f\right)=\frac{\varpi\varpi_{d-1}}{c\varpi_d}
\int_{-c}^cf(t)\left(1-\frac{t^2}{c^2}\right)^{\frac d2-1}\,dt.
\end{eqnarray}
\end{itemize}
\end{theorem}
\begin{proof}
(1). 
Applying (\ref{downa}) 
repeatedly, we get (\ref{meanlh}):
\begin{eqnarray*}
\int_{\bf S}\frh^{r-l}\left(L^{\bf t}_{\bf u}\cap X\right)\,d{\bf
u}\phantom{xxxxxxxxxxxxxxxxxxxxxxxxxxxxxxxxxxxxx}\\
=\!\!\int\limits_{\cS_2\times\dots\times\cS_l}
\!\!\Big(\int_{\cS_1}\frh^{r-l} \left(L^{t_1}_{u_1}\mathop\cap
\left(L_{u_2}^{t_2}\mathop\cap\dots\mathop\cap L^{t_l}_{u_l}\cap
X\right)\right)\,du_1\Big)\,du_2\dots du_l\\
=\frac{\varpi_{r-l}}{\varpi_{r-l+1}}s_1
\left(1-\frac{t_1^2}{c_1^2}\right)^{\frac{d_1-1}2}
\kern-14pt\int\limits_{\cS_2\times\dots\times\cS_l}\!\!\!
\frh^{r-l+1}\left(L_{u_2}^{t_2}\mathop\cap\dots\mathop\cap
L^{t_l}_{u_l}\cap X\right)\,du_2\dots du_l\\
=\dots=\frac{\varpi_{r-l}}{\varpi_{r}}\frh^r(X)
\prod_{i=1}^ls_i\left(1-\frac{t_i^2}{c_i^2}\right)^{\frac{d_i-1}2}.
\end{eqnarray*}

(2). Similarly, by (\ref{ununa}),
\begin{eqnarray*}
\int_{\bf S}\frh^{m}\left(U^{\bf t}_{\bf u}\cap X\right)\,d{\bf
u}= \frac{\vk_{d_1}\left(\frac{t_1}{c_1}\right)}{\varpi_{d_1}}
\int\limits_{\cS_2\times\dots\times\cS_l}\!\!\!
\frh^{m}\left(U_{u_2}^{t_2}\mathop\cap\dots\mathop\cap
U^{t_l}_{u_l}\right)\,du_2\dots du_l=\dots\\
=\varpi\prod_{i=1}^l\frac{\vk_{d_i}\left(\frac{t_i}{c_i}\right)}
{{\varpi_{d_i}}}.
\end{eqnarray*}
This proves (\ref{meanuh}).

(3). The equality (\ref{meauu}) is a particular case of
(\ref{meanuh}) for $r=m$ and $X=M$.  By (\ref{meauu}), the
right-hand side of (\ref{lmean}) is equal to
$-\frac{d}{dt}\sfM\left(\frh^m(U_u^t)\right)$. According to
(\ref{defle}), we have to check the equality
\begin{eqnarray}\label{mleem}
\int_\cS\frac{d}{dt}\frh^m(U_u^t)\,du=\frac{d}{dt}\sfM(\frh^m(U^{t}_u)).
\end{eqnarray}
for almost all $t$.  We claim that for any fixed $u\in\cS$ the
function $\frh^{m}\left(U^t_u\right)$ is absolutely continuous on
$t$ in the interval $u(M)$. Indeed, on the set
$u(M)\setminus\CV(u)$ this is true according to Lemma~\ref{lecon},
({\romannumeral2}); since  this function is non-increasing and
$\CV(u)$ is finite, it is sufficient to prove that it is
continuous. We have
\begin{eqnarray*}
\lim_{\ep\to0}\frh^m\left(U^{t-\ep}_u\setminus
U^{t+\ep}_u\right)=\frh^m(L^t_u)=0,
\end{eqnarray*}
where the first equality is evident and the second holds because
$u$ is real analytic on $M$:  $\frh^m(L^t_u)>0$ implies $u=t$,
contradictory to the assumption in ($\sfE$) that $\cE$ is
orthogonal to constant functions since $u\neq0$ due to the
inclusion $u\in\cS$. (The implication is obvious if $m=1$; for
$m>1$ one can use the immersion $\iota$ to prove that a set of
positive $\frh^m$-measure in $M$ intersects sufficiently many real
analytic curves (preimages of the big circles) in sets of positive
$\frh^1$-measure.)

Thus, $\frh^m\left(U^t_u\right)$ is absolutely continuous on
$[-c,c]$ and we may apply the Newton--Leibnitz formula on $t$ to
$\frl(L_u^t)$ on any subinterval $[a,b]$ of $[-c,c]$. Since it is
nonnegative and has variation $\varpi$ on $[-c,c]$, $\frl(L_u^t)$
is summable on $\cS\times[-c,c]$. In particular,
$\sfM(\frl(L^t_u))$ is well defined for almost all $t$. By
Fubini's theorem,
\begin{eqnarray}\label{mldif}
\int_a^b\sfM\left(\frl(L_u^t)\right)\,dt
=\sfM\Big(\int_a^b\frl(L_u^t)\,dt\Big)
=\sfM(\frh^m(U_u^a))-\sfM(\frh^m(U_u^b)).
\end{eqnarray}
Thus, the integrals of the left-hand and the right-hand parts of
(\ref{mleem}) over any subinterval in $[-c,c]$  coincide; hence,
(\ref{mleem}) holds almost everywhere on $t$. This proves
(\ref{lmean}).

Since 
$\frh^m\left(U^t_u\right)$ is absolutely continuous, we may apply
Fubini's theorem  to the equality
\begin{eqnarray*}
\sfM(I_f)=\int_\cS\Big(\int_{-c}^cf(t)\frl(L^t_u)\,dt\Big)\,du.
\end{eqnarray*}
Thus, (\ref{fmean}) follows from (\ref{lmean}).
\end{proof}

\begin{remark}\rm
Note that the right-hand side of (\ref{meauu}), as well as
(\ref{ununa}) and (\ref{meanuh}), depends only on $d$ and
$\varpi$. Thus, for isotropy irreducible homogeneous spaces, the
expectations of $\frh^m(U_u^{t})$ are independent of their
topology and of the spectrum of $\De$ in $\cE$. According to
(\ref{defle}), the same is true for the Leray measure. The
corollary below shows that the asymptotic behavior of the
expectations as $\dim\cE\to\infty$ is independent of the choice of
subspaces $\cE\subset L^2(M)$, as well as of $M$, except for the
left-hand side of (\ref{dinfl}).\qed
\end{remark}
\begin{corollary}\label{limdi}
Let $\cE_n$ be a sequence of subspaces of $L^2(M)$ which satisfy
{\rm($\sfE$)}. Suppose $d_n\to\infty$ as $n\to\infty$. Then for
any $t\in\bbR$ and $r$-rectifiable $X\subseteq M$, where $r\leq
m$, we have
\begin{eqnarray}
\lim_{n\to\infty}\frac1{s_n}\sfM(\frh^{r-1}(L_u^{t}\cap X))=
\frac{\varpi_{r-1}}{\varpi_{r}}\,\frh^r(X)\,e^{-\frac{t^2}{2}},
\label{dinfl}\\
\lim_{n\to\infty}\sfM(\frh^r(U_u^{t}\cap X))=
\frac{\frh^r(X)}{\sqrt{\pi}}\erfc\left(\frac{t}{\sqrt2}\right),
\label{dinfu}\\
\lim_{n\to\infty}\sfM(\frl(L_u^t))=
\frac{\varpi}{\sqrt{2\pi}}\,e^{-\frac{t^2}{2}}, \label{asyle}
\end{eqnarray}
where $\erfc(t)=\int_t^\infty e^{-\tau^2}\,d\tau$. If $f$ is a
piecewise continuous function on $\bbR$ such that
$\int_{-\infty}^\infty|f(t)|e^{-\frac{t^2}{2}}\,dt<\infty$, then
\begin{eqnarray}\label{dinil}
\lim_{n\to\infty}\sfM\left(I_f\right)=\frac{\varpi}{\sqrt{2\pi}}
\int_{-\infty}^\infty f(t)e^{-\frac{t^2}{2}}\,dt.
\end{eqnarray}
\end{corollary}
\begin{proof}
Since $c_n^2=d_n+1$, we have $\lim_{n\to\infty}
\left(1-\frac{t^2}{c_n^2}\right)^{\frac{d_n-1}2}=
e^{-\frac{t^2}{2}}$. Together with the theorem for $l=1$, this
implies (\ref{dinfl}) and (\ref{dinfu}). Taking in account the
equality
\begin{eqnarray*}
\lim_{n\to\infty}\frac{\varpi_{d_n-1}}{c_n\varpi_{d_n}}=
\lim_{n\to\infty}\frac{\Ga\left(\frac{d_n+1}{2}\right)}
{\sqrt{\pi}\,c_n\,\Ga\left(\frac{d_n}{2}\right)}=\frac1{\sqrt{2\pi}},
\end{eqnarray*}
we get (\ref{asyle}) and (\ref{dinil}).
\end{proof}
The results above make it possible to find expectations for radial
distributions on $\cE$. Let $\al$ be a nonnegative measurable
function on $[0,\infty)$ such that $\al\neq0$ and
\begin{eqnarray*}
a_k=\int_0^\infty r^k\al(r)\,dr<\infty
\end{eqnarray*}
for all $k\in\bbN$. It defines a probability measure
$\al_d(x)\,dx$ on $\cE$, where $dx$ stands for the Lebesgue
measure on $\cE$, with the density
\begin{eqnarray}\label{noral}
\al_d(x)=\frac1{a_d\varpi_d}\al(|x|).
\end{eqnarray}
We denote the mean value of a function $f$ on $\cE$ with respect
to a probability measure $\eta(x)\,dx$ as
\begin{eqnarray*}
\sfM^\eta(f)=\int_\cE f(x)\,\eta(x)\,dx.
\end{eqnarray*}
Since $U^t_u\cup U^{-t}_{-u}=M$, $U^t_u\cap U^{-t}_{-u}=L^t_u$,
and $u$ is real analytic, for $u\neq0$ we have
\begin{eqnarray*}
\frh^m\left(U^t_u\right)+\frh^m\left(U^{-t}_{-u}\right)=\varpi.
\end{eqnarray*}
Hence we may assume $t\geq0$.
\begin{proposition}\label{merad}
Let $\al,a_d$, and $\al_d$ be as above and $t\geq0$. Then
\begin{eqnarray}\label{meaal}
\sfM^{\al_d}\left(\frh^{m-1}\left(L^{ct}_{u}\right)\right)
=\frac{\varpi\varpi_{m-1}s}{a_d\varpi_m}
\int_t^\infty(r^2-t^2)^{\frac{d-1}{2}}r\al(r)\,dr,\\
\sfM^{\al_d}\left(\frh^{m}(U_u^{ct})\right)
=\frac{\varpi\varpi_{d-1}}{2a_d\varpi_d}\int_t^\infty\Big(\int_{0}^\infty
\tau^{\frac{d}{2}-1}\,\al\left(\sqrt{\tau+\xi^2}\right)\,d\tau\Big)\,d\xi,
\phantom{x}\label{meald}\\
\sfM^{\al_d}\left(\frl(L_u^{ct})\right)
=\frac{\varpi\varpi_{d-1}}{2ca_d\varpi_d}\int_{0}^\infty
\tau^{\frac{d}{2}-1}\,\al\left(\sqrt{\tau+t^2}\right)\,d\tau.
\label{meler}
\end{eqnarray}
\end{proposition}
\begin{proof}
Let $\cS_r$ denote the sphere of radius $r$ centered at zero (thus
$\cS=\cS_1$) and $du$ be the invariant probability measure on
$\cS_r$. We have
\begin{eqnarray}\label{polar}
\int_\cE\frh^{m-1}\left(L^{ct}_u\right)\al(|x|)\,dx
=\varpi_d\int_0^\infty
\left(\int_{\cS_r}\frh^{m-1}\left(L^{ct}_u\right)\,du\right)r^d\al(r)\,dr.
\end{eqnarray}
Clearly,
\begin{eqnarray}\label{srtos}
\int_{\cS_r}f(u)\,du= \int_{\cS}f(ru)\,du
\end{eqnarray}
for any continuous function $f$ on $\cE$. Furthermore,
$L^t_u=L^{rt}_{ru}$ for all $r>0$. If $|u|<t$, then
$L^{ct}_u=\varnothing$ by (\ref{leqc}). Thus, the integral in the
right-hand side of (\ref{polar}) is equal to
\begin{eqnarray*}
\int_t^\infty
\left(\int_{\cS_r}\frh^{m-1}\left(L^{ct}_{u}\right)\,du\right)r^d\al(r)\,dr
=\int_t^\infty
\left(\int_{\cS}\frh^{m-1}\left(L^{\frac{ct}{r}}_u\right)\,du\right)r^d\al(r)\,dr
\\
=\int_t^\infty\sfM\left(\frh^{m-1}\big(L^{\frac{ct}{r}}_u\big)\right)
r^d\al(r)\,dr.
\end{eqnarray*}
Using (\ref{mealu}), we obtain (\ref{meaal}) by a straightforward
calculation.
The expectations of $\frh^m\left(U^{ct}_u\right)$ and
$\frl\left(L^{ct}_u\right)$ can be calculated similarly: since
$U_u^t=U_{ru}^{rt}$ for any $r>0$ and $U_u^{ct}=\varnothing$ if
$|u|<t$, (\ref{srtos}) and (\ref{meauu}) imply
\begin{eqnarray*}
\sfM^{\al_d}\left(\frh^{m}(U_u^{ct})\right)
=\frac1{a_d}\int_t^\infty\Big(\int_{\cS}
\frh^{m}(U_u^{\frac{ct}{r}})\,du\Big) r^d\al(r)\,dr\phantom{xxxxxxxx}\\
=\frac{\varpi\varpi_{d-1}}{a_d\varpi_d}\int_t^\infty\Big(\int_{\frac
tr}^1\left(1-\eta^2\right)^{\frac{d}{2}-1}\,d\eta\Big)\,r^d\al(r)\,dr.
\end{eqnarray*}
Let us change the order of integration and substitute
$\xi=\frac{\eta}{r}$:
\begin{eqnarray*}
\int_t^\infty\Big(\int_{t}^r
\left(r^2-\xi^2\right)^{\frac{d}{2}-1}\,d\xi\Big)\,r\al(r)\,dr
=\int_t^\infty\Big(\int_{\xi}^\infty
\left(r^2-\xi^2\right)^{\frac{d}{2}-1}\,r\al(r)\,dr\Big)\,d\xi.
\end{eqnarray*}
Using the change of variable $\tau=r^2-\xi^2$, we get
(\ref{meald}) and, by differentiation of (\ref{meald}) on $t$,
(\ref{meler}).
\end{proof}
\begin{corollary}\label{gauex}
{ If $\al_d$ is the Gaussian density on $\cE$ defined by {\rm
(\ref{noral})} for $\al(t)=G^\si(t)=e^{-\frac{t^2}{\si^2}}$, then}
\begin{eqnarray*}
\sfM^{\al_d}\left(\frh^{m-1}\left(L^{ct}_{u}\right)\right)=
\varpi\frac{\varpi_{m-1}}{\varpi_m}se^{-\frac{t^2}{\si^2}},\\
\sfM^{\al_d}\left(\frh^{m}\left(U^{ct}_{u}\right)\right)
=\frac{\varpi}{\sqrt\pi}\erfc\left(\frac{t}{\si}\right),\\
\sfM^{\al_d}\left(\frl\left(L^{ct}_{u}\right)\right)=
\frac{\varpi}{c\si\sqrt\pi}e^{-\frac{t^2}{\si^2}},
\end{eqnarray*}
where $\erfc(t)=\int_t^\infty e^{-\tau^2}\,d\tau$.
\end{corollary}
\begin{proof}
We have $a_d=
\frac{\si^{d+1}}2\Ga\left(\frac{d+1}{2}\right)$. By (\ref{meaal}),
\begin{eqnarray*}
\sfM^{\al_d}\left(\frh^{m-1}\left(L^{ct}_{u}\right)\right)
=\frac{\varpi\varpi_{m-1}s}{a_d\varpi_m}
\int_t^\infty\left(r^2-t^2\right)^{\frac{d-1}{2}}
re^{-\frac{r^2}{\si^2}}\,dr\phantom{xxxxxxxxxxxx}\\
=\frac{\varpi\varpi_{m-1}s}{
\Ga\left(\frac{d+1}{2}\right)\varpi_m}e^{-\frac{t^2}{\si^2}}
\int_0^\infty\tau^{\frac{d-1}{2}}e^{-\tau}\,d\tau
=\varpi\frac{\varpi_{m-1}}{\varpi_m}se^{-\frac{t^2}{\si^2}},
\end{eqnarray*}
where $\tau=\frac{r^2-t^2}{\si^2}$. Further,
$\frac{\varpi_{d-1}}{a_d\varpi_d}
=\frac2{\sqrt\pi\si^{d+1}\Ga\left(\frac d2\right)}$; therefore,
\begin{eqnarray*}
\sfM^{\al_d}\left(\frh^{m-1}\left(U^{ct}_{u}\right)\right)
=\frac{\varpi}{\si^{d+1}\sqrt\pi\Ga\left(\frac
d2\right)}\int_t^\infty\Big(\int_0^\infty
\tau^{\frac{d}{2}-1}e^{-\frac{\tau}{\si^2}}\,
d\tau\Big) e^{-\frac{\xi^2}{\si^2}}\,d\xi\\
=\frac{\varpi}{\si\sqrt\pi\Ga\left(\frac
d2\right)}\int_t^\infty\Big(\int_0^\infty
\eta^{\frac{d}{2}-1}e^{-\eta}\,d\eta\Big)
e^{-\frac{\xi^2}{\si^2}}\,d\xi
=\frac{\varpi}{\sqrt\pi}\erfc\left(\frac{t}{\si}\right).
\end{eqnarray*}
Differentiating on $t$, we get the last equality of the corollary.
\end{proof}
\begin{remark}\rm
According to (\ref{asyle}), $\lim_{d\to\infty}\sfM(\frl(N_u))=
\frac{\varpi}{\sqrt{2\pi}}$. In the papers \cite{ORW} and
\cite{W09}, the expectations of $\frl(N_u)$ were computed for tori
$\bbR^n/\bbZ^n$ and spheres $S^m$ with the Gaussian distribution
in $\cE$ normalized by the condition that for any fixed $p\in M$
the average of $|u(p)|^2$ is equal to $1$. In both cases, the
expectation is independent of $\dim\cE$ and equals to
$\frac{\varpi}{\sqrt{2\pi}}$, where $\varpi=1$ for $\bbR^n/\bbZ^n$
and $\varpi=\varpi_m$ for $S^m$.  By a direct computation one can
check that this is equivalent to the relation $\si c=\sqrt2$ in
the notation of Corollary~\ref{gauex}. It follows from
Corollary~\ref{gauex} that the same is true for all isotropy
irreducible homogeneous spaces: the expectation of $\frl(N_u)$ for
the Gaussian distribution with this normalization is equal to
$\frac{\varpi}{\sqrt{2\pi}}$ independently of $\cE$. For the
uniform distribution on spheres the expectation depends on
$\dim\cE$ but mildly since
$\frac{\varpi_{d-1}}{c\varpi_d}\to\frac{1}{\sqrt{2\pi}}$ as
$d\to\infty$.\qed
\end{remark}

\section{Upper bounds for the expectations of $L^p$ norms}
We use the symbol $a$ instead of the standard $p$ in $\|u\|_a$
(the norm in the spaces $L^a(M)$, $1\leq a\leq\infty$). There are
two reasons for it: first, $p$ denotes a point of $M$ in the text
above, and second, we do not exclude the cases $a\in(0,1)$ and
even $a\in(-1,0)$ in the calculation below.
\begin{theorem}\label{exmua}
Let $a>-1$. The function $|u|^a$ is integrable on $M$ for almost
all $u\in\cS$. Moreover, $\int_M|u(p)|^a\,dp\in
L^1\left(\cS\right)$ and
\begin{eqnarray}\label{taexe}
\sfM\left(\int_M|u(p)|^a\,dp\right)=
\frac{\Ga\left(\frac{a+1}{2}\right)
\Ga\left(\frac{d+1}{2}\right)(d+1)^{\frac{a}2}}
{\sqrt{\pi}\,\,\Ga\left(\frac{a+d+1}{2}\right)}.
\end{eqnarray}
If $a>2$ or $a\in(-1,0)$, then for all $d\in\bbN$
\begin{eqnarray*}
\sfM\left(\int_M|u(p)|^a\,dp\right)<
2^{\frac{a}{2}}\frac{\Ga\left(\frac{a+1}{2}\right)}{\sqrt\pi},
\end{eqnarray*}
the reverse inequality holds for $a\in(0,2)$, and the equality is
true if $a=0$ or $a=2$.
\end{theorem}
\begin{proof}
If $a>0$, then we may apply (\ref{fmean}) to $f(t)=|t|^a$: {
\begin{eqnarray*}
\sfM(I_f)=
c^a\frac{\varpi_{d-1}}{\varpi_d}
\int_{-1}^1|t|^a\left(1-t^2\right)^{\frac{d}{2}-1}\,dt
=c^a\frac{\varpi_{d-1}}{\varpi_d}\int_0^1\tau^{\frac{a-1}{2}}
\left(1-\tau\right)^{\frac{d}{2}-1}\,d\tau\\
=c^a\frac{\varpi_{d-1}}{\varpi_d}B\left(\frac{a+1}2,\frac
d2\right)= c^a\frac{\varpi_{d-1}}{\varpi_d}
\frac{\Ga\left(\frac{d}{2}\right)\Ga\left(\frac{a+1}{2}\right)}
{\Ga\left(\frac{a+d+1}{2}\right)}
=\frac{\Ga\left(\frac{a+1}{2}\right)\Ga\left(\frac{d+1}{2}\right)}{\sqrt\pi\,
\Ga\left(\frac{a+d+1}{2}\right)} c^a .
\end{eqnarray*}
All equalities, with the possible exception for the first, hold
true for $a\in(-1,0)$.} Hence, we get the same result for such $a$
approximating $|t|^a$ by the functions $f_n(t)=\min\{|t|^a,n\}$.
Indeed, the sequence $I_{f_n}(u)$ increases for any $u\in\cS$ and
$f_n(t)$ converges to $|t|^a$ if $t\neq0$. Hence
\begin{eqnarray*}
\frac{c\varpi_d}{\varpi\varpi_{d-1}}\sfM\left(I_{f_n}\right)=
\int_{-c}^cf_n(t)\left(1-\frac{t^2}{c^2}\right)^{\frac d2-1}\,dt
\to\int_{-c}^c|t|^a\left(1-\frac{t^2}{c^2}\right)^{\frac d2-1}\,dt
\end{eqnarray*}
as $n\to\infty$. It follows from Levy's and Lebesgue's theorems
that $I_{f_n}(u)\to\int_M|u(p)|^a\,dp$ and $|u|^a\in
L^1\left(M\right)$ for almost all $u\in\cS$. Thus,
$\int_M|u(p)|^a\,dp\in L^1\left(\cS\right)$. This verifies the
calculation. The inequalities follow from Lemma~\ref{gamma} below,
the cases $a=0$ and $a=2$ are obvious.
\end{proof}
For the sake of completeness, we give a proof of some properties
of Euler's function $\Ga$.

\begin{lemma}\label{gamma}
Set $\vf_b(t)=\frac{t^b\Ga(t)}{\Ga(t+b)}$,
$f(t)=\ln\left(\left(\frac{e}{t}\right)^{t-\frac12}\Ga(t)\right)$.
\begin{enumerate}
\item[\rm{(a)}] The function $\vf_b$ decreases on $(0,\infty)$ if
~$0<b<1$ and increases if ~$b>1$. Moreover, if $b<0$, then $\vf_b$
increases on $(-b,\infty)$. For any $b\in\bbR$
$\lim_{t\to\infty}\vf_b(t)=1$.
\item[\rm{(b)}] The function $f$ is convex on $(0,\infty)$ and
$\lim_{t\to\infty}f(t)=\ln\sqrt{\frac{2\pi}{e}}$.
\item[\rm{(c)}] For all $t\in(\frac12,\infty)$
\begin{eqnarray}\label{inega}
1>\left(\frac{e}{t}\right)^{t-\frac12}\frac{\Ga(t)}{\sqrt{\pi}}
>\sqrt{\frac{2}{e}}.
\end{eqnarray}
\end{enumerate}
\end{lemma}
\begin{proof}
The limit in (a) follows from the Stirling formula. To prove the
first and the second assertions in (a), let us consider
$\Psi(x)=\frac{d}{dx}\ln\Ga(x)$. We have
\begin{eqnarray*}
\Psi''(x)=-2\sum_{k=0}^\infty\frac1{(x+k)^3}<0
\end{eqnarray*}
for all $x>0$. Hence the function
\begin{eqnarray*}
\eta_t(b)=\frac{d}{dt}\ln\vf_b(t)=\frac{b}{t}+\Psi(t)-\Psi(t+b)
\end{eqnarray*}
is strictly convex on $(-t,\infty)$ for any fixed $t>0$. The
evident equalities
\begin{eqnarray*}
\eta_t(0)=\eta_t(1)=0
\end{eqnarray*}
imply $\eta_t(b)<0$ for $b\in(0,1)$ and $\eta_t(b)>0$ if
$b\in(1,\infty)$ or $b\in(-t,0)$. This proves (a).

The computation of limit in (b) is standard. Differentiating $f$
we get $f''(t)=\Psi'(t)-\frac{1}{t}-\frac1{2t^2}$, where
$\Psi'(t)=\frac{d^2}{dt^2}\ln\Ga(t)=\sum_{n=0}^\infty
\frac{1}{(t+n)^2}$. We have
\begin{eqnarray*}
\Psi'(t)=
\frac1{2t^2}+\frac12\sum_{n=0}^\infty
\left(\frac1{(t+n)^2}+\frac1{(t+n+1)^2}\right)
>\frac1{2t^2}+\sum_{n=0}^\infty\int_n^{n+1}\frac{d\tau}{(t+\tau)^2}\\
= \frac1{2t^2}+\int_0^\infty\frac{d\tau}{(t+\tau)^2}=
\frac{1}{t}+\frac1{2t^2},
\end{eqnarray*}
where the inequality holds since the function $\frac{1}{t^2}$ is
strictly convex. It follows that $f''(t)>0$ on $(0,\infty)$. Thus,
(b) is true.

The function $f$ decreases since it is convex and has a finite
limit at infinity. Therefore,
\begin{eqnarray*}
f\left(\frac12\right)=\ln\sqrt{\pi}>f(t)>\ln\sqrt{\frac{2\pi}{e}}
=\lim_{\tau\to\infty}f(\tau)
\end{eqnarray*}
for all $t$ in $(\frac12,\infty)$. This proves (c).
\end{proof}

Let $\frh_p$ denote the Lie algebra of the stable subgroup of
$p\in M$ and $\pi_p$ be the orthogonal projection in $\cE$ onto
$T_{\bar p}\bar M=d_p\iota\left(T_pM\right)$.
\begin{lemma}\label{lipsc}
All functions $u\in\cS$ are Lipschitz with the coefficient
$\ka=cs$. Moreover, this coefficient is attained if and only if
$u=\frac{1}{\ka}\xi\phi_p$ for some $p\in M$ and $\xi\in\frg$ such
that $\xi\perp\frh_p$, $|\xi|=1$.
\end{lemma}
\begin{proof}
For any $u\in\cS$ and $p\in M$ we have $u(p)=c\scal{u}{\bar p}$.
Since the mapping $p\to\bar p$ is a local metric homothety with
the coefficient $s$ and the linear function
$\ell_u(v)=\scal{u}{v}$ is Lipschitz in $\cE$ with the coefficient
$1$, the first assertion follows. Furthermore, the gradient of the
restriction of $\ell_u$ onto $\bar M$ may be identified with
$\pi_pu$. Hence, $\max_{q\in M}|\nabla u(q)|=\ka$ if and only if
$u\in T_{\bar p}\bar M$ for some $p\in M$. This happens if and
only if $u$ is proportional to $\xi\phi_p$ for some $\xi\in\frg$
and $|u|=1$. A description of such $u$ is given in the statement
of the lemma.
\end{proof}
Let $B(p,r)=\{q\in M:\,\rho(q,p)<r\}$ be the ball with respect to
the Riemannian distance $\rho$ in $M$. Clearly, there exist $b>0$
and $r_0>0$, which depend only on the geometry of $M$, such that
\begin{eqnarray}\label{volba}
\frh^m(B(p,r))>b\varpi r^m
\end{eqnarray}
for all $r\in(0,r_0)$.
\begin{lemma}\label{lplip}
Let $u$ be a $\ka$-Lipschitz function on $M$, $a\geq1$, and $b$,
$r_0$ be as above. Then
\begin{eqnarray}\label{inkkr}
\|u\|_\infty\leq b^{-\frac1a}r^{-\frac{m}{a}}\|u\|_a+\ka r.
\end{eqnarray}
 for all $r\in(0,r_0)$.
\end{lemma}
\begin{proof}
We may assume that $\|u\|_\infty=\sup_{p\in M}u(p)$ replacing $u$
with $-u$ if necessary. According to the Chebyshev inequality,
$\frac1{\varpi}\frh^m(U^t_u)\leq\frac{\|u\|_a^a}{t^a}$ for all
$t>0$. Hence, if $p\in M$ and
\begin{eqnarray}\label{checo}
\frac{\|u\|_a^a}{t^a}<\frac1{\varpi}\frh^m(B(p,r)),
\end{eqnarray}
then the ball $B(p,r)$ contains a point $q\notin U^t_u$. We have
$\rho(p,q)<r$ and $u(q)<t$; it follows that $u(p)<t+\ka r$.
Moreover, $\|u\|_\infty<t+\ka r$ since $t$ and $r$ are independent
of $p$. If $\|u\|_a^at^{-a}=br^m$ and $r<r_0$, then (\ref{checo})
is true due to (\ref{volba}). This proves (\ref{inkkr}).
\end{proof}

\begin{theorem}\label{lpinf}
Let $\cE_n$ be a sequence of subspaces of $L^2(M)$ which satisfy
{\rm($\sfE$)}. Suppose $d_n\to\infty$ as $n\to\infty$. Then
\begin{eqnarray}\label{limupa}
\lim_{n\to\infty}\sfM\left(\int_M|u(p)|^a\,dp\right)
=\frac{2^{\frac{a}{2}}}{\sqrt{\pi}}\,\Ga\left(\frac{a+1}{2}\right),
&&a>-1,\\
\limsup_{n\to\infty}\sfM\left(\|u\|_a\right)\leq
\sqrt{2}\,\pi^{-\frac{1}{2a}}\Ga\left(\frac{a+1}{2}
\right)^{\frac1a} < \sqrt{\frac{a+1}e},&&a\geq1.\label{limsu}
\end{eqnarray}
Moreover, for any space $\cE$ satisfying
{\rm($\sfE$)}
\begin{eqnarray}
\sfM\left(\|u\|_a\right)<\sqrt{\frac{a+1}e},&&a\geq2,\label{lesss}\\
\sfM\left(\|u\|_\infty\right)<K(\sqrt{\ln\ka}+1),\label{logs}
\end{eqnarray}
where $\ka=cs$ and $K>0$ is independent of $\cE$.
\end{theorem}
\begin{proof}
The equality (\ref{limupa}) follows from Theorem~\ref{exmua} and
Lemma~\ref{gamma}, (a), with $t=\frac{d+1}{2}$ and
$b=\frac{a}{2}$.

If $a>2$, then $\vf_b(t)<1$ due to Lemma~\ref{gamma}, (a); this
proves the inequality
\begin{eqnarray*}
\sfM\left(\|u\|_a^a\right)<
\frac{2^{\frac{a}{2}}}{\sqrt{\pi}}\,\Ga\left(\frac{a+1}{2}\right).
\end{eqnarray*}
Set $t=\frac{a+1}{2}$. Then $t-\frac12=\frac{a}2$. By
Lemma~\ref{gamma}, (a), for all $a>0$ we have
\begin{eqnarray}\label{g1est}
\left(\frac{2}{e}\right)^{\frac1{2a}}\sqrt{\frac{a+1}{e}}<
\left(\frac{2^{\frac{a}{2}}}{\sqrt\pi}
\Ga\left(\frac{a+1}{2}\right)\right)^{\frac{1}{a}}
<\sqrt{\frac{a+1}{e}}.
\end{eqnarray}
Since $\sfM\left(\|u\|_2\right)=1<\sqrt{\frac3e}$, we get
(\ref{lesss}). Moreover, (\ref{g1est}) and (\ref{limupa}) imply
(\ref{limsu}) for $a\geq1$:
\begin{eqnarray*}
\limsup_{n\to\infty}\sfM(\|u\|_a)\leq
\limsup_{n\to\infty}\sfM(\|u\|^a_a)^{\frac1a}=
\left(\frac{2^{\frac{a}{2}}}{\sqrt\pi}
\Ga\left(\frac{a+1}{2}\right)\right)^{\frac{1}{a}}<\sqrt{\frac{a+1}{e}}.
\end{eqnarray*}

Due to Lemma~\ref{lipsc}, we may use Lemma~\ref{lplip} with
$\ka=cs$. Setting $r=\frac1{\ka}$ and assuming $\ka$ sufficiently
large, we get
\begin{eqnarray*}
\|u\|_\infty\leq b^{-\frac1a}\ka^{\frac{m}{a}}\|u\|_a+1.
\end{eqnarray*}
If $a=\ln k$, then $\ka^{\frac{m}{a}}=e^{m}$ and
$b^{-\frac1a}\leq\max\{1,b^{-1}\}$. Integrating over $\cS$ and
using (\ref{lesss}), we get $\sfM(\|u\|_\infty)\leq K\sqrt{\ln
\ka}+1$ with some $K>0$.
\end{proof}
We may assume $b$ arbitrary close to 1 making $r$ smaller if
necessary. Thus, the inequality holds for any $K>e^{m-\frac12}$ if
$\ka$ is sufficiently large.

\vskip1cm
\vbox{V.M. Gichev\\
Omsk Branch of Sobolev Institute of Mathematics\\
Pevtsova, 13, 644099, Omsk, Russia\\
gichev@ofim.oscsbras.ru}
\end{document}